\documentclass[11pt]{amsart}

\usepackage[a4paper]{geometry}
\usepackage{latexsym, amsthm, amsfonts, amsmath,amssymb,graphicx,lmodern}

 \usepackage{hyperref}
 \usepackage{color}

\hypersetup{urlcolor=blue, citecolor=red}

\usepackage[all]{xy}
\usepackage[latin1]{inputenc} 
\usepackage[T1]{fontenc} 
\usepackage[english]{babel} 

\newtheorem{theo}{Theorem}
\newtheorem{lem}[theo]{Lemma}
\newtheorem{propo}[theo]{Proposition}

\theoremstyle{definition}
\newtheorem{defi}[theo]{Definition}

\theoremstyle{remark}
\newtheorem{rem}[theo]{Remark}
\newtheorem{ex}[theo]{Example}

\makeatletter

\def\R{\mathbb{R}}
\def\Z{\mathbb{Z}}
\def\C{\mathbb{C}}
\def\N{\mathbb{N}}
\def\Q{\mathbb{Q}}

\def\n{\eta}
\def\r{\rho}

\def\a{\alpha}

\def\f{\varphi}

\def\b{\beta}

\def\n'{\nu}

\def\g{\gamma}

\begin{document}
\sloppy
\title{Real difference Galois theory.}
\author{Thomas Dreyfus}
\address{Universit\'e Claude Bernard Lyon 1, Institut Camille Jordan, 
43 boulevard du 11 novembre 1918,
69622 Villeurbanne, France.}
\email{dreyfus@math.univ-lyon1.fr}
\thanks{Work supported by the labex CIMI. This project has received funding from the European Research Council (ERC) under the European Union's Horizon 2020 research and innovation programme under the Grant Agreement No 648132.}



\date{\today}

\begin{abstract}
In this paper, we develop a difference Galois theory in the setting of real fields. After proving the existence and uniqueness of the real Picard-Vessiot extension, we define the real difference Galois group and prove a Galois correspondence. 
\end{abstract} 
\maketitle

\tableofcontents

\section*{Introduction}

Let us consider an equation of the form:
\begin{equation}\label{eq5}
\phi Y=AY,
\end{equation} 
where $A$ is an invertible matrix having coefficients in a convenient field $\mathbf{k}$\footnote{In all the paper, all fields are of characteristic zero.} and $\phi$ is an automorphism of $\mathbf{k}$. A typical example is $\mathbf{k}:=\C(x)$ and $\phi y(x):=y(x+1)$. The aim of the difference Galois theory is to study (\ref{eq5}) from an algebraic point of view. See \cite{VdPS97} for details on this theory. See also \cite{BiaBi,Fr63,HS,Morik,MorUm}. The classical framework for difference Galois theory is to assume that $C$, the subfield of $\mathbf{k}$ of elements invariant under $\phi$, is algebraically closed. The goal of the present paper is to present a descent result. We explain what happens if we take instead a smaller field $\mathbf{k}$, such that $\mathbf{k}$ is a real field and $C$ is real closed, see $\S \ref{sec2}$ for the definitions. \\\par 
Assume that $C$ is algebraically closed and let us make a brief summary of the difference Galois theory. An important object attached to (\ref{eq5}) is the Picard-Vessiot extension. Roughly speaking, a Picard-Vessiot extension is a ring extension of $\mathbf{k}$ containing a basis of solutions of (\ref{eq5}). The Picard-Vessiot extension always exists, but the uniqueness is proved in \cite{VdPS97} only in the case where $C$ is algebraically closed. To the Picard-Vessiot extension, we attach a group, the difference Galois group, that measures the algebraic relations between solutions belonging to the Picard-Vessiot extension. This group may be seen as a linear algebraic subgroup of invertible matrices in coefficients in $C$. We also have a Galois correspondence. Note that several definitions of the difference Galois group have been made and the comparison between different Galois groups can be found in~\cite{CHS}. \pagebreak[3]\\\par 
From now, we drop the assumption that $C$ is algebraically closed, and we make the assumptions that $\mathbf{k}$ is a real field and $C$ is real closed. Our approach will follow \cite{CHS13,CHvdPa}, which prove similar results in the framework of differential Galois theory. Let us present a rough statement of our main result, Theorem \ref{theo1}. We prove that in this setting, a real Picard-Vessiot exists, i.e., a Picard-Vessiot extension that is additionally a real ring. Then, we also show a uniqueness result: given $R_{1}$ and $R_{2}$ two real Picard-Vessiot extensions, then $R_{1}$ and $R_{2}$ are isomorphic over $\mathbf{k}$ if and only if $R_{1}\otimes_{\mathbf{k}}R_{2}$ has no elements $x$ satisfying $x^{2}+1=0$. We define a real difference Galois group, which may be seen as a linear algebraic subgroup of invertible matrices in coefficients in the algebraic closure of $C$, and that is defined over $C$. See Proposition~\ref{propo2}. This allows us to prove a Galois correspondence, see Theorem~\ref{theo2}. See also \cite{CH13b,Dyc} for similar results in the framework of differential Galois theory.\\\par
 
The paper is presented as follows. In $\S \ref{sec1}$, we make some reminders of difference algebra. In $\S \ref{sec2}$, we state and prove our main result, Theorem \ref{theo1}, about the existence and uniqueness of real Picard-Vessiot extensions. In $\S \ref{sec3}$, we define the real difference Galois group, and prove a Galois correspondence.\\\par 

\textbf{Acknowledgments.} The author would like the thank the anonymous referee for permitting him to correct some mistakes that was originally made in the paper.
\pagebreak[3]

\section{Reminders of difference algebra}\label{sec1}
For more details on what follows, we refer to \cite{Cohn}. A difference ring $(R,\phi)$ is a ring $R$ together with a ring automorphism $\phi : R \rightarrow R$. An ideal of $R$ stabilized by $\phi$ is called a difference ideal of $(R,\phi)$. A simple difference ring $(R,\phi)$ is a difference ring with only difference ideals $(0)$ and $R$. If $R$ is a field then $(R,\phi)$ is called a difference field.\par 
Let $(R,\phi)$ be a difference ring and $m\in \N^{*}$. The difference ring~$R\{X_{1},\dots,X_{m}\}_{\phi}$ of difference polynomials in~$m$ indeterminacies over~$R$ is the usual polynomial ring in the infinite set of variables
$$\{ \phi^{\nu}(X_{j})\}^{\nu\in \Z}_{j\leq m},~$$
and with automorphism extending the one on $R$ defined by:
$$\phi\left(\phi^{\nu}(X_{j}) \right) =\phi^{\nu+1}(X_{j}).$$
The ring of constants $R^{\phi}$ of the difference ring $(R,\phi)$ is defined by 
$$R^{\phi}:=\{f \in R \ | \ \phi(f)=f\}.$$
If $R^{\phi}$ is a field, the ring of constants will be called field of constants. \\ \par 
 
A difference ring morphism from the difference ring $(R,\phi)$ to the difference ring $(\widetilde{R},\widetilde{\phi})$ is a ring morphism $\varphi : R \rightarrow \widetilde{R}$ such that 
$\varphi \circ \phi = \widetilde{\phi} \circ \varphi$.

A difference ring $(\widetilde{R},\widetilde{\phi})$ is a difference ring extension of a difference ring $(R,\phi)$ 
if $\widetilde{R}$ is a ring extension of $R$ and $\widetilde{\phi}_{\vert R}=\phi$;
in this case, we will often denote $\widetilde{\phi}$ by $\phi$. Two difference ring extensions $(\widetilde{R}_{1},\widetilde{\phi}_{1})$ and $(\widetilde{R}_{2},\widetilde{\phi}_{2})$ of a difference ring $(R,\phi)$ are isomorphic over $(R,\phi)$ if there exists a difference ring isomorphism $\varphi$ from $(\widetilde{R}_{1},\widetilde{\phi}_{1})$ to $(\widetilde{R}_{2},\widetilde{\phi}_{2})$ such that 
$\varphi_{\vert R}=\operatorname{Id}$.\\\par 
 Let $(R,\phi)$ be a difference ring such that $X^{2}+1\in R[X]$ is irreducible, i.e., there is no $x\in R$ such that $x^{2}+1=0$. We define, $R[i]$, to be the ring ${R[i]:=R[X]/(X^{2}+1)}$. We equip $R[i]$ with a structure of difference ring with $\phi (i)=i$. If $(R,\phi)$ is a difference ring with an element $x\in R$ satisfying $x^{2}+1=0$, we make the convention that $R[i]=R$.

\pagebreak[3]
\section{Existence and uniqueness of Picard-Vessiot extensions over real fields}\label{sec2}

Let $(\mathbf{k},\phi)$ be a difference field of characteristic zero. Consider a linear difference system 
\begin{equation}\label{eq2}
\phi Y=AY \hbox{ with } A \in \mathrm{GL}_{n}(\mathbf{k}),
\end{equation} 
where $\mathrm{GL}_{n}$ denotes the group of invertible $n\times n$ square matrices with entries in $\mathbf{k}$.

\pagebreak[3]
\begin{defi}
A Picard-Vessiot extension for (\ref{eq2}) over $(\mathbf{k},\phi)$ is a difference ring extension $(R,\phi)$ of $(\mathbf{k},\phi)$ such that 
\begin{itemize}
\item[(1)] there exists $U \in \mathrm{GL}_{n}(R)$ such that $\phi (U) = AU$ (such a $U$ is called a fundamental matrix of solutions of (\ref{eq2}));
\item[(2)] $R$ is generated, as a $\mathbf{k}$-algebra, by the entries of $U$ and $\det(U)^{-1}$;
\item[(3)] $(R,\phi)$ is a simple difference ring.\\
\end{itemize}
\end{defi}

We may always construct a Picard-Vessiot extension as follows. Take an indeterminate $n\times n$ square matrix $X:=X_{j,k}$ and consider $\mathbf{k}\{X,\det(X)^{-1}\}_{\phi}$ which is equipped with a structure of difference ring with $\phi X=AX$. Then, for any $I$, maximal difference ideal of $\mathbf{k}\{X,\det(X)^{-1}\}_{\phi}$, the ring $\mathbf{k}\{X,\det(X)^{-1}\}_{\phi}/I$ is a simple difference ring and therefore, is a Picard-Vessiot extension. \par 
According to \cite[$\S$1.1]{VdPS97}, when the field of constants $C:=\mathbf{k}^{\phi}$ is algebraically closed, we also have the uniqueness of the Picard-Vessiot extension, up to a difference ring isomorphism. Furthermore, in this case we have $C=R^{\phi}$ and, see \cite[Corollary~1.16]{VdPS97}, there exist an idempotent $e\in R$, and $t\in \N^{*}$, such that $\phi^{t}(e)=e$, $R=\displaystyle \bigoplus_{j=0}^{t-1}\phi^{j}(e)R$, and for all $0\leq j\leq t-1$, $\phi^{j}(e)R$ is an integral domain. \\\par 

In \cite{CHS}, it is defined the notion of weak Picard-Vessiot extension we will need in the next section.

\pagebreak[3]
\begin{defi}
A weak Picard-Vessiot extension for (\ref{eq2}) over $(\mathbf{k},\phi)$ is a difference ring extension $(R,\phi)$ of $(\mathbf{k},\phi)$ such that 
\begin{itemize}
\item[(1)] there exists $U \in \mathrm{GL}_{n}(R)$ such that $\phi (U) = AU$;
\item[(2)] $R$ is generated, as a $\mathbf{k}$-algebra, by the entries of $U$ and $\det(U)^{-1}$;
\item[(3)] $R^{\phi}=\mathbf{k}^{\phi}=C$.\\
\end{itemize}
\end{defi}

From what is above, we deduce that when the field of constants is algebraically closed, a Picard-Vessiot extension is a weak Picard-Vessiot extension. Note that the converse is not true as shows \cite[Example 1.25]{VdPS97}. 
\\\par 

We say that a field $\mathbf{k}$ is real when $0$ is not a sum of squares in $\mathbf{k}\setminus \{0\}$. We say that a field $\mathbf{k}$ is real closed when $\mathbf{k}$ does not admit an algebraic extension that is real. In particular, $\mathbf{k}$ is real closed if and only if $\mathbf{k}[i]$ is algebraically closed and satisfies $\mathbf{k}[i]\neq \mathbf{k}$.\par  

\pagebreak[3]
\begin{ex}
The field $\R((x))$ of formal Laurent series with real coefficients is real. The field $\Q(x)$ is real. The field of real numbers is real closed. 
\end{ex}

{\it From now we assume that $\mathbf{k}$ is a real field and its field of constants $C:=\mathbf{k}^{\phi}$ is real closed.} \\

Remind that we have seen that we have the existence of  $(R,\phi)$, Picard-Vessiot extension for (\ref{eq2}) over $(\mathbf{k},\phi)$.

\pagebreak[3]
\begin{lem}\label{lem2}
Let $(R,\phi)$, be a Picard-Vessiot extension for (\ref{eq2}) over $(\mathbf{k},\phi)$ and assume that $R\neq R[i]$. Then, $(R[i],\phi)$, is a Picard-Vessiot extension for (\ref{eq2}) over $(\mathbf{k}[i],\phi)$.  \end{lem}

\begin{proof}
Let $(0)\neq I$ be a difference ideal of $(R[i],\phi)$. Note that $I\cap R$ is a difference ideal of $(R,\phi)$. We claim that $I\cap R\neq (0)$.  Let $a,b\in R$ with $0\neq a+ib\in I$. Then, $\phi(a)+i\phi(b)\in I$ and for all $c\in R$, $ac+ibc\in I$. Let $J$ be the smallest difference ideal of $R$ that contains $a$. From what is above, we may deduce that for all $a_{1}\in J$, there exists $b_{1}\in R$ such that $a_{1}+ib_{1}\in I$. Since $(R,\phi)$ is a simple difference ring, we have two possibilities: $J=(0)$ and $J=R$. We are going to treat separately the two cases. Assume that $J=(0)$. Then $a=0$ and $ib\in I$. But $ib\times(-i)=b\in I\cap R\setminus \{0\}$ which proves our claim when $J=(0)$. Assume that $J=R$. Then, there exists $b_{1}\in R$ such that $1+ib_{1}\in I$. But $(1+ib_{1})(1-ib_{1})=1+b_{1}^{2}\in I\cap R$. Since $R\neq R[i]$ we find that $1+b_{1}^{2}\neq 0$ which proves our claim when $J=R$.\par 
Since $I\cap R\neq (0)$ and $(R,\phi)$ is a simple difference ring, $I\cap R=R$. We now remark that $I$ is stable by multiplication by $\mathbf{k}[i]$, which shows that $I=R[i]$. This proves the lemma.
\end{proof}

\pagebreak[3]
\begin{propo}\label{propo3}
Let $(R,\phi)$, be a Picard-Vessiot extension for (\ref{eq2}) over $(\mathbf{k},\phi)$. Then, there exist an idempotent $e\in R$, and $t\in \N^{*}$, such that $\phi^{t}(e)=e$, $R=\displaystyle \bigoplus_{j=0}^{t-1}\phi^{j}(e)R$, and for all $0\leq j\leq t-1$, $\phi^{j}(e)R$ is an integral domain. \end{propo}

\begin{proof}
Let us treat separately two cases. Assume that $R\neq R[i]$. Due to Lemma \ref{lem2}, $(R[i],\phi)$, is a Picard-Vessiot extension for (\ref{eq2}) over $(\mathbf{k}[i],\phi)$. We remind that by definition, if $R\neq R[i]$, we extend $\phi$ to $R[i]$ by $\phi (i)=i$.  
Then, the field of constants of $R[i]$ is $C[i]$, which is algebraically closed.  From \cite[Corollary~1.16]{VdPS97}, we obtain that there exist $a,b\in R$, with  $a+ib$ is idempotent, $t\in \N^{*}$, such that  ${\phi^{t}(a+ib)=a+ib}$,
\begin{equation}\label{eq3}
R[i]=\displaystyle \bigoplus_{j=0}^{t-1}\phi^{j}(a+ib)R[i],
\end{equation}
and for all $0\leq j\leq t-1$, $\phi^{j}(a+ib)R[i]$ is an integral domain. Let $e:=a^{2}+b^{2}\in R$. A straightforward computation shows that $a-ib$ is idempotent. Since $e=(a+ib)(a-ib)$ is the product of two idempotent elements it is also idempotent. Using  $\phi^{t}(a-ib)=a-ib$, we find  $\phi^{t}(e)=e$.\par 
 Let us prove that for all $0\leq j\leq t-1$, $\phi^{j}(a-ib)R[i]$ is an integral domain. Let $0\leq j\leq t-1$, $c+id\in R[i]$ with $c,d\in R$, such that $\phi^{j}(a-ib)(c+id)=0$. It follows that  $\phi^{j}(a+ib)(c-id)=0$ and therefore, $c-id=0=c+id$ since for all $0\leq j\leq t-1$, $\phi^{j}(a+ib)R[i]$ is an integral domain. We have proved that for all $0\leq j\leq t-1$, $\phi^{j}(a-ib)R[i]$ is an integral domain. Let us prove that for all $0\leq j\leq t-1$, $\phi^{j}(e)R[i]$ is an integral domain. Let $0\leq j\leq t-1$, $c\in R[i]$, such that $c\phi^{j}(e)=c\phi^{j}(a+ib)\phi^{j}(a-ib)=0$. We use successively the fact that $\phi^{j}(a+ib)R[i]$ and $\phi^{j}(a-ib)R[i]$ are integral domains to deduce that $c=0$, which shows that $\phi^{j}(e)R[i]$ is an integral domain. Therefore, for all $0\leq j\leq t-1$, $\phi^{j}(e)R$ is an integral domain. \par 
We claim that $\{\phi^{j}(e),0\leq j\leq t-1\}$ are linearly independent over $R[i]$. Let us consider ${c_{0},\dots,c_{t-1}\in R[i]}$ such that $\displaystyle\sum_{j=0}^{t-1}c_{j}\phi^{j}(e)=0$.
We have $\displaystyle\sum_{j=0}^{t-1}c_{j}\phi^{j}(a-ib)\phi^{j}(a+ib)=0$. We use (\ref{eq3}) to deduce that for all $0\leq j\leq t-1$, $c_{j}\phi^{j}(a-ib)=0$. We remind that for all $0\leq j\leq t-1$, $\phi^{j}(a-ib)R[i]$ is an integral domain. This shows that for all $0\leq j\leq t-1$, $c_{j}=0$. This proves our claim.\par  
 Using (\ref{eq3}), to prove the proposition, it is now sufficient to prove the equality

\begin{equation}\label{eq4}
\displaystyle \bigoplus_{j=0}^{t-1}\phi^{j}(a+ib)R[i]=\displaystyle \bigoplus_{j=0}^{t-1}\phi^{j}(e)R[i].
\end{equation}

The inclusion $\displaystyle \bigoplus_{j=0}^{t-1}\phi^{j}(e)R[i]\subset \displaystyle \bigoplus_{j=0}^{t-1}\phi^{j}(a+ib)R[i]$ is a direct consequence of the fact that $e=(a-ib)(a+ib)\in (a+ib)R[i]$. Let us prove the other inclusion. Let ${\a\in \displaystyle \bigoplus_{j=0}^{t-1}\phi^{j}(a+ib)R[i]}$, and define $f:=\displaystyle\prod_{j=0}^{t-1}\phi^{j}(e)$ which is invariant under $\phi$. Therefore, $fR[i]$ is a difference ideal of $R[i]$. We use $e=(a+ib)(a-ib)$ and the fact that $(a+ib)R[i]$ is an integral domain to obtain that $f\neq 0$ and $fR[i]\neq (0)$. Since $R[i]$ is a simple difference ring, the difference ideal $fR[i]$ equals to $R[i]$. This means that the ideal of $R[i]$ generated by the $\phi^{j}(f)$, $j\in \Z$, is $R[i]$. Since $\phi(f)=f$, there exists $\b\in R[i]$ such that $f\b=\a$. We again use~(\ref{eq3}) to find that ${\a=f\displaystyle\sum_{j=0}^{t-1}c_{j}\phi^{j}(a+ib)}$ for some $c_{j}\in R[i]$. Since $e=(a-ib)(a+ib)$, we may define for all $0\leq j\leq t-1$, $d_{j}:=f/\phi^{j}(a-ib)\in R[i]$. A straightforward computation shows that $\a=\displaystyle\sum_{j=0}^{t-1}c_{j}d_{j}\phi^{j}(e)$, which implies $\a\in \displaystyle \bigoplus_{j=0}^{t-1}\phi^{j}(e)R[i]$. We have proved $ \displaystyle \bigoplus_{j=0}^{t-1}\phi^{j}(a+ib)R[i]\subset\displaystyle \bigoplus_{j=0}^{t-1}\phi^{j}(e)R[i]$. If we combine with the other inclusion, we obtain (\ref{eq4}). This completes the proof in the case $R\neq R[i]$.\par 
Assume that $R=R[i]$. Since $i^{2}=-1$, we have $\phi(i)=\pm i$ and then $\phi^{2}(i)=i$. Hence, $(R,\phi^{2})$ is a ring extension of $(\mathbf{k}[i],\phi^{2})$, whose field of constants is $C[i]$, which is algebraically closed. Furthermore, by construction, it is also a Picard-Vessiot extension for ${\phi^{2}Y=\phi (A)AY}$ over $(\mathbf{k}[i],\phi^{2})$.  From \cite[Corollary~1.16]{VdPS97}, we obtain that there exist an idempotent $e\in R$, $t\in \N^{*}$, such that  ${\phi^{2t}(e)=e}$,
$$
R=\displaystyle \bigoplus_{j=0}^{t-1}\phi^{2j}(e)R,
$$
and for all $0\leq j\leq t-1$, $\phi^{2j}(e)R$ is an integral domain. If $t=1$, $R=eR$ is an integral domain, and we may take $e=1$ to have the desired decomposition of $R$. Assume that $t>1$. Using the fact that $\phi$ is an automorphism we find that for all $j\in \Z$, $\phi^{j}(e)R$ is an integral domain and $\phi^{j}(e)$ is idempotent. 
 Let $t'\in \N^{*}$ maximal such that $eR,\dots,\phi^{t'-1}(e)R$ are in direct sum. This implies that there exists $  r\in R$ with $r\phi^{t'}(e)\neq 0$, such that $r\phi^{t'}(e)\in \displaystyle \bigoplus_{j=0}^{t'-1}\phi^{j}(e)R$. We claim that $r\phi^{t'}(e)\in eR$. If $t'=1$ the claim is clear. Assume that  $t'>1$. Then, for all $0<j<t'$, we have $e\phi^{j}(e)=0$ and, since $\phi$ is an automorphism,  $\phi^{t'}(e)\phi^{j}(e)=0$. It follows that $\displaystyle \bigoplus_{j=0}^{t'-1}\phi^{j}(e)R\phi^{t'}(e)\subset eR$ and therefore, $r\phi^{t'}(e)=r(\phi^{t'}(e))^{2}\in eR$, which proves the claim in the $t'>1$ case. In particular, there exists $r'\in R$ such that $r\phi^{t'}(e)=r' e\neq 0$. 
 We use the fact that $\phi^{t'}$ is an automorphism, to find ${\phi^{t'}(r)\phi^{2t'}(e)=\phi^{t'}(r')\phi^{t'}(e)\neq 0}$. Since $\phi^{t'}(e)R$ is an integral domain and $\phi^{t'}(e)$ is idempotent, we have
$\phi^{t'}(r')r\phi^{t'}(e)=\phi^{t'}(r')\phi^{t'}(e)r\phi^{t'}(e)\neq 0$. But the latter inequality implies $\phi^{t'}(r)\phi^{2t'}(e)r'e\neq 0$. This shows that $\phi^{2t'}(e)e\neq 0$.
 Since $R=\displaystyle \bigoplus_{j=0}^{t-1}\phi^{2j}(e)R$, $\phi^{2t}(e)=e$, and 
 $\displaystyle \bigoplus_{j=0}^{t'-1}\phi^{j}(e)R$, we find  $t=t'$. With $R=\displaystyle \bigoplus_{j=0}^{t-1}\phi^{2j}(e)R$, we obtain $\phi^{t}(e)\in\displaystyle \bigoplus_{j=0}^{t-1}\phi^{2j}(e)R$.
  We remind that for all $0<j<t$, we have $e\phi^{j}(e)=0$.
Using the fact that $\phi$ is an automorphism, we obtain that for all $0<j <t$, we have $\phi^{t}(e)\phi^{2j}(e)=0$. It follows that $\displaystyle \bigoplus_{j=0}^{t-1}\phi^{2j}(e)R\phi^{t}(e)\subset eR$ and therefore, $\phi^{t}(e)=(\phi^{t}(e))^{2}\in eR$. So there exists $r'\in R$ such that $\phi^{t}(e)= er'$. But an integral domain may have only one non zero idempotent element. Since $e,\phi^{t}(e)$ are non zero idempotent and $eR$ is an integral domain we find that $e=\phi^{t}(e)$.
 In particular, $\displaystyle \bigoplus_{j=0}^{t-1}\phi^{j}(e)R$ is a difference ideal of the simple difference ring $(R,\phi)$. Since  $e\neq 0$, we find that the difference ideal is not $(0)$, proving that $R=\displaystyle \bigoplus_{j=0}^{t-1}\phi^{j}(e)R$. This completes the proof in the case $R=R[i]$.
\end{proof}

Let $R$ be a difference ring that is the direct sum of integral domains $R:=\displaystyle \bigoplus_{j=0}^{t-1}R_{j}$. We define $K$, the total ring of fractions of $R$, by $K:=\displaystyle \bigoplus_{j=0}^{t-1}K_{j}$, where for all $0\leq j\leq t-1$, $K_{j}$ is the fraction field of~$R_{j}$. 

We say that $R$ is a real ring if for all $0\leq j\leq t-1$, $K_{j}$ is a real field. Note that by \cite[Theorem 2.8]{Lam},  this is equivalent to the usual definition of a real ring, that is that $0$ is not a sum of squares in $R\setminus \{0\}$, see \cite[Definition 2.1]{Lam}. \\ \par 
The notion of Picard-Vessiot extension is not well suited in the real case. Following \cite{CHvdPa}, let us define:

\pagebreak[3]
\begin{defi}
A real Picard-Vessiot extension for (\ref{eq2}) over $(\mathbf{k},\phi)$ is a difference ring extension $(R,\phi)$ of $(\mathbf{k},\phi)$ such that 
\begin{itemize}
\item[(1)] $(R,\phi)$ is a Picard-Vessiot extension for (\ref{eq2}) over $(\mathbf{k},\phi)$;
\item[(2)] $(R,\phi)$ is a real difference ring.\\
\end{itemize}
\end{defi}

Let us remind that if $(R,\phi)$ is a difference ring such that $X^{2}+1\in R[X]$ is irreducible, then $R[i]$ is the ring ${R[i]:=R[X]/(X^{2}+1)}$. If $(R,\phi)$ is a difference ring with $x\in R$ satisfying $x^{2}+1=0$, we make the convention that $R[i]=R$. \par 
We are now able to state our main result:

\pagebreak[3]
\begin{theo}\label{theo1} Let us consider the equation (\ref{eq2}) which has coefficients in $(\mathbf{k},\phi)$.
\begin{itemize}
\item [(1)] There exists a real Picard-Vessiot extension for (\ref{eq2}) over $(\mathbf{k},\phi)$. 
\item [(2)] Let $(R,\phi)$ be a real Picard-Vessiot extension for (\ref{eq2}) over $(\mathbf{k},\phi)$. Then, $(R,\phi)$ is a weak Picard-Vessiot extension for (\ref{eq2}) over $(\mathbf{k},\phi)$, i.e., the ring of constants of $R$ is $C$.
\item [(3)] Let $(R_{1},\phi_{1})$ and $(R_{2},\phi_{2})$ be two real Picard-Vessiot extensions for (\ref{eq2}) over $(\mathbf{k},\phi)$.  Let us equip the ring $R_{1}\otimes_{\mathbf{k}} R_{2}$ with a structure of difference ring as follows: $\phi (r_{1}\otimes_{\mathbf{k}} r_{2})=\phi_{1} (r_{1})\otimes_{\mathbf{k}} \phi_{2} (r_{2})$ for $r_{j}\in R_{j}$. Then, $(R_{1},\phi_{1})$ is isomorphic to $(R_{2},\phi_{2})$ over $(\mathbf{k},\phi)$ if and only if $R_{1}\otimes_{\mathbf{k}} R_{2}\neq R_{1}\otimes_{\mathbf{k}} R_{2}[i]$.
\end{itemize}
\end{theo}

Before proving the theorem, we are going to state and prove a lemma which is inspired by a lemma of \cite{Sei58}.\\\par 

\pagebreak[3]
\begin{lem}\label{lem1}
Consider a difference field $(\mathbf{K},\phi)$ of characteristic zero that is finitely generated over $\Q$ by the elements $u_{1},\dots,u_{m}$ and let $(\mathbf{K}_{R},\phi)$ be a real difference subfield of $(\mathbf{K},\phi)$. Then, there exists  $h : \mathbf{K}\rightarrow \C$, injective morphism of fields that induces an injective morphism from $\mathbf{K}_{R}$ to $\R$. \par 
For every $1\leq j\leq m$, $k\in \Z$, let us write $c_{j,k}:=h(\phi^{k} (u_{j}))\in \C$. Then, the assignment $u_{j}\mapsto \widetilde{u}_{j}:=(c_{j,k})_{k\in \Z}$ defines (resp. induces) an injective morphism of difference fields 
 between $(\mathbf{K},\phi)$ (resp. $(\mathbf{K}_{R},\phi)$) and $(\C^{\Z},\phi_{s})$  (resp. $(\R^{\Z},\phi_{s})$), where $\phi_{s}$ denotes the shift. 
\end{lem}

\begin{proof}[Proof of Lemma \ref{lem1}]
Let us prove that there exists $h : \mathbf{K}\rightarrow \C$, injective morphism of fields. Let $t_{j}$ be a transcendental basis of $\mathbf{K}|\Q$. Since $\mathbf{K}|\Q$ is generated as a field by a countable number of elements, the number of elements in  the transcendental basis is countable. Using the fact that $\R$ is not countable, we find that there exists $h : \Q(t_{j})\rightarrow \R$, injective morphism of fields. Let us extend $h$ to $\mathbf{K}$. 
The elements of $\mathbf{K}|\Q(t_{j})$, satisfy a list of algebraic equations, which have a solution in an extension of $\C$. Since $\C$ is algebraically closed, we find that the equations have a solution in  $\C$. In other words, we have the existence of an embedding of $\mathbf{K}$ into $\C$. \par 
 Let us prove that 
$\mathbb{K}_{R}:=h(\mathbf{K}_{R})\subset \R$. Let $t_{j}$ be a transcendental basis of $\mathbf{K}_{R}|\Q$. We have $h(\Q(t_{j})) \subset \mathbb{K}_{R}\subset \C$. Since $h$ is an injective morphism of fields and $\mathbf{K}_{R}$ is a real field, we find that $\mathbb{K}_{R}$ is a real field. Then, we obtain that the real closure of $h(\Q(t_{j}))$ contains  $\mathbb{K}_{R}$. Since by construction $h(\Q(t_{j}))\subset \R$ we find that the real closure of $h(\Q(t_{j}))$ is contained in $\R$. Then, we conclude that $ \mathbb{K}_{R}\subset  \R \subset \C$.\par 

Let $P\in \Q\{X_{1},\dots,X_{m}\}_{\phi}$. We have the following equality $P(\widetilde{u}_{1},\dots,\widetilde{u}_{m})=P((c_{1,k})_{k\in \Z},\dots,(c_{m,k})_{k\in \Z}).$ Therefore,  $P(u_{1},\dots,u_{m})=0$ if and only if $P(\widetilde{u}_{1},\dots,\widetilde{u}_{m})=0$. This shows that  the assignment $u_{j}\mapsto \widetilde{u}_{j}:=(c_{j,k})_{k\in \Z}$ defines (resp. induces) an injective morphism of difference fields 
 between $(\mathbf{K},\phi)$ (resp. $(\mathbf{K}_{R},\phi)$) and $(\C^{\Z},\phi_{s})$ (resp. $(\R^{\Z},\phi_{s})$).
\end{proof}

\begin{proof}[Proof of Theorem \ref{theo1}]
\begin{trivlist}
\item \; (1)
Let us prove the existence of a real Picard-Vessiot extension. We have seen how to construct $(R,\phi)$, Picard-Vessiot extension for (\ref{eq2}) over $(\mathbf{k},\phi)$.  Let ${U\in \mathrm{GL}_{n}(R)}$ be a fundamental solution.  
As we can see in Proposition \ref{propo3}, $R$ is a direct sum of integral domains and we may define $K$, the total ring of fractions of $R$. The ring $K$ is a direct sum of fields $K:=\displaystyle \bigoplus_{j=0}^{t-1}K_{j}$ satisfying $\phi(K_{j})=K_{j+1}$, $K_{t}:=K_{0}$. Therefore, for all $0\leq j\leq t-1$, $(K_{j},\phi^{t})$ is a difference field. Let $(\mathcal{K},\phi)$ be the difference subring of $(K,\phi)$ generated over $\Q$ by the components on the $K_{j}$, $0\leq j\leq t-1$, of the entries of $U$, $\det (U)^{-1}$, and the elements in $\mathbf{k}$ involved in the algebraic difference relations between the entries of $U$ and $\det (U)^{-1}$. In particular,  the entries of the matrix $A$ of (\ref{eq2}) belong to $\mathcal{K}$. As we can see from Lemma~\ref{lem1}, for all $0\leq j\leq t-1$, there exists $\widetilde{h}_{j}$, an embedding of $(\mathcal{K}\cap K_{j},\phi^{t})$ into $(\C^{\Z},\phi_{s})$. If $t>1$, without loss of generality, we may assume that for all $0\leq j\leq t-2$, (resp. for $j=t-1$), for all $u\in \mathcal{K}\cap K_{j}$, $\widetilde{h}_{j}(u)=\widetilde{h}_{j+1}(\phi (u))$ (resp. $\phi_{s}\left(\widetilde{h}_{0}(\phi (u))\right)=\widetilde{h}_{t-1}(u)$).
We may define $\widetilde{h}$, an embedding of the difference ring $(\mathcal{K},\phi)$ into the difference ring $(\C^{\Z},\phi_{s})$ as follows. Let $k=\sum_{j=0}^{t-1}k_{j}$ with $k\in\mathcal{K}$, $k_{j}\in  K_{j}$ and let us define $\widetilde{h}(k)\in \C^{\Z}$ as the sequence which term number $c+dt$, with $0\leq c\leq t-1$, $d\in \Z$,  equals to the term number $d$ of $\widetilde{h}_{c}(k_{c})$.  Furthermore, since $\mathbf{k}$ is a real field, we find, see Lemma \ref{lem1}, that for all $k\in (\mathcal{K}\cap \mathbf{k},\phi)$, $\widetilde{h}(k)\in \R^{\Z}$.\par 
 Let $C_{r,1},\dots,C_{r,n}$, (resp. $C_{i,1},\dots,C_{i,n}$) be the real parts (resp. the imaginary parts) of the columns of the term number zero of $\widetilde{h}(U)$. We remind that $U$ is invertible. Therefore, the term number zero of $\widetilde{h}(U)$ is invertible. Then, we may extract $n$ columns ${C_{1},\dots,C_{n}\in \{C_{r,1},\dots,C_{r,n},C_{i,1},\dots,C_{i,n}\}}$, that are linearly independent.
 Therefore, there exists $B$ a matrix with entries in $\Q[i]$, such that the term number zero of $\widetilde{h}(U)B$ has columns $C_{1},\dots,C_{n}$, and is consequently real and invertible. Then, the term number zero of $\widetilde{h}(U)B$ is invertible. Since the term number zero of $\widetilde{h}(U)$ is also invertible, we find that $B\in \mathrm{GL}_{n}(\Q[i])$.
 Let $V:=UB$ be a fundamental solution, which belongs to $\mathrm{GL}_{n}(\mathcal{K}[i])$.  The map $\widetilde{h}$ extends to a morphism of difference rings between $(\mathcal{K}[i],\phi)$ and $(\C^{\Z},\phi_{s})$. Consequently, we have $\widetilde{V}:=\widetilde{h}(V)=\widetilde{h}(U)\widetilde{h}(B)$. \par 
Let $(\widetilde{\Q},\phi_{s})$ be the difference subring of $(\Q^{\Z},\phi_{s})$ of constant sequences. Note that $(\widetilde{\Q},\phi_{s})$ is a difference field. Let $(\widetilde{\mathbf{k}},\phi_{s})$ be the difference subring of $(\R^{\Z},\phi_{s})$ generated over $(\widetilde{\Q},\phi_{s})$, by the elements $\widetilde{h}(k)$, $k\in (\mathcal{K} \cap\mathbf{k},\phi)$. Note that $(\widetilde{\mathbf{k}},\phi_{s})$ is a difference field. We remind that since $\mathbf{k}$ is a real field, Lemma \ref{lem1} tells us that $\widetilde{h}(A)\in (\mathrm{GL}_{n}(\R))^{\Z}$. Since the term number zero of $\widetilde{V}$ belongs to $ \mathrm{GL}_{n}(\R)$, and $\phi_{s} (\widetilde{V})=\widetilde{h}(A)\widetilde{V}$, we obtain that $\widetilde{V}\in (\mathrm{GL}_{n}(\R))^{\Z}$.
Let $(\widetilde{R},\phi_{s})$ be the difference subring of $(\R^{\Z},\phi_{s})$ generated over $(\widetilde{\mathbf{k}},\phi_{s})$ by the entries of $\widetilde{V}$, and $\det(\widetilde{V})^{-1}$.\par 

 We claim that $(\widetilde{R},\phi_{s})$ is a simple difference ring. 
 To the contrary, assume that there exists  $I$, a difference ideal of $\widetilde{R}$  different from $(0)$ and $\widetilde{R}$. It follows that $I(\widetilde{R}[i])$ is different from $(0)$ and $\widetilde{R}[i]$. We have a natural embedding from $(\widetilde{R},\phi_{s})$ into $ (R[i],\phi)$.  Then, $I(\widetilde{R}[i])$ induces a difference ideal of $(R[i],\phi)$, which is different from $(0)$ and $R[i]$. Let us treat separately two cases.  If $R=R[i]$, then we use the fact that the Picard-Vessiot extension $(R[i],\phi)$ is a simple difference ring to conclude that we have a contradiction and $(\widetilde{R},\phi_{s})$ is a simple difference ring. If $R\neq R[i]$, we use Lemma \ref{lem2}, to deduce that $(R[i],\phi)$ is a Picard-Vessiot extension for $(\ref{eq2})$ over $(\mathbf{k}[i],\phi)$ and therefore, is a simple difference ring.  We find a contradiction and we have proved  our claim, that is that $(\widetilde{R},\phi_{s})$ is a simple difference ring. We additionally use the fact that by construction $\widetilde{R}$ is a real ring to prove that $(\widetilde{R},\phi_{s})$  is a real Picard-Vessiot extension for $\phi \widetilde{Y}=\widetilde{h}(A)\widetilde{Y}$, over $(\widetilde{\mathbf{k}},\phi_{s})$. \par 
 Let $R_{1}$, be the difference ring generated over  $\mathcal{K} \cap\mathbf{k}$ by the entries of the fundamental solution $V$ and $\det(V)^{-1}$.
  Using the fact that $(\widetilde{\mathbf{k}},\phi_{s})$ is isomorphic to $( \mathcal{K} \cap\mathbf{k},\phi)$, and $(\widetilde{R},\phi_{s})$  is a real Picard-Vessiot extension for $\phi \widetilde{Y}=\widetilde{h}(A)\widetilde{Y}$, over $(\widetilde{\mathbf{k}},\phi_{s})$, we obtain that $R_{1}$ is a real Picard-Vessiot extension for (\ref{eq2}) over $(\mathcal{K} \cap\mathbf{k},\phi)$. \par 
 Let us prove that $R_{2}:=R_{1}\otimes_{\mathcal{K} \cap\mathbf{k}}\mathbf{k}$ is a real Picard-Vessiot extension for (\ref{eq2}) over $(\mathbf{k},\phi)$. By construction, $R_{2}$ is generated over $\mathbf{k}$ by the entries of $V$ and $\det(V)^{-1}$. It is sufficient to prove that $R_{2}$ is a simple difference ring which is real. \par 
 We claim that $R_{2}$ is a real ring.
   Let $a_{j}\in R_{2}$ such that $\sum_{j} (a_{j})^{2}=0$. Let us write $a_{j}=\sum_{\ell}r_{j,\ell}\otimes_{\mathcal{K} \cap\mathbf{k}} k_{j,\ell}$, with $r_{j,\ell}\in R_{1}$, $k_{j,\ell}\in \mathbf{k}$. With $\sum_{j} (a_{j})^{2}=0$, we obtain an algebraic relation over $\Q$ between the $r_{j,\ell}$ and the $k_{j,\ell}$. Note that $R_{1}\subset R_{2}\subset K[i]$. Then, as a consequence of the definition of $\mathcal{K}$, we find that $\mathcal{K}[i]$ is the difference subring of $(K[i],\phi)$ generated over $\Q$ by  $i$, the components on the $K_{j}[i]$, $0\leq j\leq t-1$, of the entries of $V$, $\det (V)^{-1}$, and the elements in $\mathbf{k}$ involved in the algebraic difference relations between the entries of $V$ and $\det (V)^{-1}$. Furthermore, since $i$ in an algebraic number that does not belong to the real field $\mathbf{k}$, and $V=UB$, with 
 $U\in \mathrm{GL}_{n}(\mathcal{K})$, $B\in \mathrm{GL}_{n}(\Q[i])$, we find that 
   the elements in $\mathbf{k}$ involved in the latter relations  are in fact involved in algebraic difference relations between the entries of $U$ and $\det (U)^{-1}$, proving that they belong to $\mathcal{K}$.
 Hence, we find that for all $j,\ell$, $r_{j,\ell}\in R_{1}\cap \mathcal{K}[i]$ and $k_{j,\ell}\in \mathbf{k}\cap \mathcal{K}$. 
Therefore, for all $j$, $a_{j}\in R_{1}$.
    Since $R_{1}$ is a real ring, we find that for all $j$, $a_{j}=0$, proving that $R_{2}$ is a real ring.\par 
     It is now sufficient to prove that $(R_{2},\phi)$ is a simple difference ring. Let $I\neq (0)$ be a difference ideal of $(R_{2},\phi)$.  Since $V=UB$, with $V\in \mathrm{GL}_{n}(R_{2})$, $U\in \mathrm{GL}_{n}(R)$, and $B\in \mathrm{GL}_{n}(\Q[i])$, we find that $R_{2}[i]=R[i]$. The difference ideal $I$ induces the difference ideal $I[i]$ of $(R[i],\phi)$. Let us treat separately two cases. If $R[i]=R$, then $(R[i],\phi)$ is a simple difference ring since it is a Picard-Vessiot extension for (\ref{eq2}) over $(\mathbf{k},\phi)$, proving that $I[i]=R[i]$ and $I=R_{2}$. Therefore,  $(R_{2},\phi)$ is a simple difference ring and a real Picard-Vessiot extension for (\ref{eq2}) over $(\mathbf{k},\phi)$. Assume that $R[i]\neq R$.
 With Lemma~\ref{lem2}, $(R[i],\phi)$ is a Picard-Vessiot extension for (\ref{eq2}) over $(\mathbf{k}[i],\phi)$. Then, $(R[i],\phi)$ is a simple difference ring and $I[i]=R_{2}[i]$, proving that $I=R_{2}$. This shows that  $(R_{2},\phi)$ is a simple difference ring and a real Picard-Vessiot extension for (\ref{eq2}) over $(\mathbf{k},\phi)$. \\ \par 

\item (2)  With Lemma \ref{lem2} we find that $(R[i],\phi)$ is  a Picard-Vessiot extension for (\ref{eq2}) over $(\mathbf{k}[i],\phi)$. 
Remind that by assumption, $C[i]$ is algebraically closed. As we can deduce from \cite[Lemma~1.8]{VdPS97}, $R[i]^{\phi}=C[i]$. It follows that $R^{\phi}\subset C[i]$. By assumption, $R$ is a real ring. This implies that $i\notin R$. Therefore, $C=\mathbf{k}^{\phi}\subset R^{\phi}$.  Hence, the field of constants of $R$ is $C$. \\\par 
 \item (3)
Let us assume that, $R_{1}\otimes_{\mathbf{k}} R_{2}\neq R_{1}\otimes_{\mathbf{k}} R_{2}[i]$ and let us prove that $(R_{1},\phi_{1})$ is isomorphic to $(R_{2},\phi_{2})$ over $(\mathbf{k},\phi)$. We remind, see Lemma \ref{lem2}, that for $j\in \{1,2\}$, $(R_{j}[i],\phi_{j})$, is a Picard-Vessiot extension for (\ref{eq2}) over $(\mathbf{k}[i],\phi)$. We also remind that the field of constants of $\mathbf{k}[i]$ is $C[i]$. Due to \cite[Proposition 1.9]{VdPS97}, we find that $(R_{1}[i],\phi_{1})$ is isomorphic to $(R_{2}[i],\phi_{2})$ over $(\mathbf{k}[i],\phi)$. Let $\varphi : R_{1}\rightarrow R_{2}[i]$ be the restriction of the morphism. Then, we may define a morphism of difference rings 
$$\begin{array}{llll}
\Psi :& R_{1}\otimes_{\mathbf{k}} R_{2}&\rightarrow & R_{2}[i]\\
&x\otimes y&\mapsto &\varphi  (x)y.
\end{array} $$
The morphism $\Psi$ is a $R_{2}$-linear map, and the image of $R_{1}\otimes_{\mathbf{k}} R_{2}$ under $\Psi$ is a $R_{2}$-submodule of $R_{2}[i]$, called $V$.
\par 
  The assumption $R_{1}\otimes_{\mathbf{k}} R_{2}\neq R_{1}\otimes_{\mathbf{k}} R_{2}[i]$ implies that there are no $f\in R_{1}\otimes_{\mathbf{k}} R_{2}$ such that $f^{2}+1=0$. Since $\Psi$ is a morphism of difference ring, there are no $g\in V$ such that $g^{2}+1=0$, which proves $i\notin V$. Combining this fact  with the inclusion $R_{2}\subset V$, we obtain that $V=R_{2}$ (we remind that $V$ is a $R_{2}$-submodule of $R_{2}[i]$). In other words, the image of $R_{1}$ under $\varphi$ is included in $R_{2}$. This implies that  $(R_{1},\phi_{1})$ is isomorphic to $(R_{2},\phi_{2})$ over $(\mathbf{k},\phi)$.\par 
Conversely, if $(R_{1},\phi_{1})$ is isomorphic to $(R_{2},\phi_{2})$ over $(\mathbf{k},\phi)$, then there exists a morphism of difference rings $\varphi : R_{1}\rightarrow R_{2}$. As above, let us define $\Psi$, morphism of difference rings between $R_{1}\otimes_{\mathbf{k}} R_{2}$ and $R_{2}$ defined by $\Psi (x\otimes y)=\varphi(x)y$. Since $R_{2}$ is a real ring, we find that $R_{2}\neq R_{2}[i]$. Since $\Psi$ is a morphism of difference rings, we obtain that $R_{1}\otimes_{\mathbf{k}} R_{2}\neq R_{1}\otimes_{\mathbf{k}} R_{2}[i]$.
\end{trivlist}
\end{proof}

The following example, who is inspired by \cite{CHS13}, illustrates a situation where two Picard-Vessiot extensions are not isomorphic. 

\begin{ex}\label{ex1}
Let $\phi := f(z)\mapsto f(2z)$ and consider  ${\phi Y=\sqrt{2}Y}$ which has coefficients in $\R(x)$. Let us consider the following fundamental solutions $(\sqrt{x})$ and $(i\sqrt{x})$. Consider the corresponding difference ring extensions  ${R_{1}|\R(x):=\R\left[\sqrt{x},\sqrt{x}^{-1}\right]|\R(x)}$ and ${R_{2}|\R(x):=\R\left[i\sqrt{x},(i\sqrt{x})^{-1}\right]}$. Let us prove that $(R_{1},\phi)$ is a simple difference ring. The proof for $(R_{2},\phi)$ is similar. Let $I\neq (0)$ be a difference ideal of $R_{1}$ and let $P\in \R[X]$ with minimal degree such that $P(\sqrt{x})\in I$. Let $k\in \N$ be the degree of $P$. Assume that $k\neq 0$. We have $\phi (P(\sqrt{x}))=P(\sqrt{2}\sqrt{x})\in I$, which shows that $\phi (P(\sqrt{x}))-\sqrt{2}^{k}P(\sqrt{x})=Q(\sqrt{x})\in I$ where $Q\in \R[X]$ has degree less than $k$. This is in contradiction with the minimality of $k$, and  shows that $k=0$. This implies that $I=R_{1}$, which proves that $(R_{1},\phi)$ is a simple difference ring.
Since $R_{1}$ and $R_{2}$ are real rings, $R_{1}|\R(x)$ and $R_{2}|\R(x)$ are two real Picard-Vessiot extensions for $\phi Y=\sqrt{2}Y$ over $(\R(x),\phi)$. 
Note that there are no difference ring isomorphism between $(R_{1},\phi)$ and $(R_{2},\phi)$ over $\R(x)$ because $X^{2}=x$ has a solution in $R_{1}$ and no solutions in $R_{2}$. This is not in contradiction with Theorem \ref{theo1}  since $R_{1}\otimes_{\R(x)}R_{2}= R_{1}\otimes_{\R(x)}R_{2}[i]$, because $$\left(\sqrt{x}\otimes_{\R(x)}\frac{1}{i\sqrt{x}}\right)^{2}=-1.$$ 
\end{ex}
\pagebreak[3]
\section{Real difference Galois group}\label{sec3}
In this section, we still consider (\ref{eq2}). Let $(R,\phi)$ be a real Picard-Vessiot extension for (\ref{eq2}) over $(\mathbf{k},\phi)$ with fundamental solution $U\in \mathrm{GL}_{n}(R)$. Consider the difference ring $(R[i],\phi)$, which is different from $(R,\phi)$, since $R$ is a real ring. Inspiriting from \cite{CHS13}, let us define the real difference Galois group as follows:

\pagebreak[3]
\begin{defi}
We define  $G_{R[i]}$,  as the group of difference ring automorphism of $R[i]$ letting $\mathbf{k}[i]$ invariant. We define $G$, the real difference Galois group of (\ref{eq2}), as the group $\{ \f_{|R}, \f \in G_{R[i]}\}$.
\end{defi}
Note that elements of $G$ are maps from $R$ to $R[i]$.
Due to Theorem \ref{theo1}, (2), we have an injective group morphism
$$\begin{array}{cccc}
\r_{U}: & G & \longrightarrow & \mathrm{GL}_{n}(C[i]) \\ 
 & \f &\longmapsto  & U^{-1}\f(U),
\end{array}$$
which depends on the choice of the fundamental solution $U$ in $R$. Another choice of a fundamental solution in $R$ will gives a representation that is conjugated to the first one.  \\ \par 

Remind, see Proposition \ref{propo3}, that there exist an idempotent $e\in R$, and $t\in \N^{*}$, such that $\phi^{t}(e)=e$, $R=\displaystyle \bigoplus_{j=0}^{t-1}\phi^{j}(e)R$, and for all $0\leq j\leq t-1$, $\phi^{j}(e)R$ is an integral domain. Due to Lemma \ref{lem2}, $(R[i],\phi)$ is a Picard-Vessiot extension for (\ref{eq2}) over $(\mathbf{k}[i],\phi)$. Furthermore, $R[i]=\displaystyle \bigoplus_{j=0}^{t-1}\phi^{j}(e)R[i]$ and the total ring of fractions of $R[i]$ equals $K[i]$, where $K$ is the total ring of fractions of $R$. Then, we call $G_{K[i]}$, the classical difference Galois group of (\ref{eq2}), the group of difference ring automorphism of $K[i]$ letting $\mathbf{k}[i]$ invariant. See \cite{VdPS97} for more details. The difference Galois group of (\ref{eq2}) may also be seen as a subgroup of $\mathrm{GL}_{n}(C[i])$. Furthermore, its image in $\mathrm{GL}_{n}(C[i])$ is a linear algebraic subgroup of $\mathrm{GL}_{n}(C[i])$. We have the following result in the real case.

\pagebreak[3]
\begin{propo}\label{propo2}
Let $(R,\phi)$ be a real Picard-Vessiot extension for (\ref{eq2}) over $(\mathbf{k},\phi)$ with fundamental solution $U\in \mathrm{GL}_{n}(R)$. Let $G$, be the real difference Galois group of (\ref{eq2}) and $G_{K[i]}$, be the difference Galois group of (\ref{eq2}). We have the following equality
$$\hbox{Im } \r_{U}=\left\{ U^{-1}\f(U), \f \in G\right\}=\left\{ U^{-1}\f(U), \f \in G_{K[i]}\right\}.$$ 
Furthermore, $\hbox{Im } \r_{U}$ is a linear algebraic subgroup of~$\mathrm{GL}_{n} (C[i])$ defined over $C$. We will identify~$G$ with a linear algebraic subgroup of~$\mathrm{GL}_{n} (C[i])$ defined over $C$ for a chosen fundamental solution.
\end{propo}

\begin{proof}
Let us prove the equality $\left\{ U^{-1}\f(U), \f \in G\right\}=\left\{ U^{-1}\f(U), \f \in G_{K[i]}\right\}$. Remind that $U\in \mathrm{GL}_{n}(R)$. Since an element of $G_{K[i]}$ induces an element of $G$, we obtain the inclusion $\left\{ U^{-1}\f(U), \f \in G_{K[i]}\right\}\subset \left\{ U^{-1}\f(U), \f \in G\right\}$. Let $\f\in G$. We may extend $\f$ as an element $\f_{K[i]}\in G_{K[i]}$ by putting $\f_{K[i]}(i)=i$ and for all $0\leq j\leq t-1$, $a,b\in \phi^{j}(e)R[i]$, $\f_{K[i]}(\frac{a}{b})=\frac{\f_{K[i]} (a)}{\f_{K[i]} (b)}$. Since $U\in \mathrm{GL}_{n}(R)$, we find that $U^{-1}\f(U)=U^{-1}\f_{K[i]}(U)$. Therefore, we obtain the other inclusion $\left\{ U^{-1}\f(U), \f \in G\right\}\subset\left\{ U^{-1}\f(U), \f \in G_{K[i]}\right\}$ and the equality $\left\{ U^{-1}\f(U), \f \in G\right\}=\left\{ U^{-1}\f(U), \f \in G_{K[i]}\right\}$.\\ \par 
With a similar reasoning to what is above, we obtain the equalities: 
$$\left\{ U^{-1}\f(U), \f \in G\right\}=\left\{ U^{-1}\f(U), \f \in G_{K[i]}\right\}=\left\{ U^{-1}\f(U), \f \in G_{R[i]}\right\}.$$ 
We define $G_{R}$, as the group of difference ring automorphism of $R$ letting $\mathbf{k}$ invariant.  Due to Theorem \ref{theo1}, (2), $(R,\phi)$ is a weak Picard-Vessiot extension for (\ref{eq2}) over $(\mathbf{k},\phi)$. Applying \cite[Proposition 2.2]{CHS}, we find that $\left\{ U^{-1}\f(U), \f \in G_{R}\right\}$ is a linear algebraic subgroup of~$\mathrm{GL}_{n} (C)$. 
  Then, we may use \cite[Corollary 2.5]{CHS}, to find that the latter group, viewed as a linear algebraic subgroup of~$\mathrm{GL}_{n} (C[i])$, equals to $\left\{ U^{-1}\f(U), \f \in G_{R[i]}\right\}$. We conclude the proof using the equality
$$\left\{ U^{-1}\f(U), \f \in G\right\}=\left\{ U^{-1}\f(U), \f \in G_{R[i]}\right\}.$$ 
\end{proof}

We finish this section by giving the Galois correspondence. See \cite[Theorem 1.29]{VdPS97} for the analogous statement in the case where $C$ is algebraically closed. 

\pagebreak[3]
\begin{theo}\label{theo2}
Let $(R,\phi)$ be a real Picard-Vessiot extension for (\ref{eq2}) over $(\mathbf{k},\phi)$ with total ring of fractions $K$, $\mathcal{F}$ be the set of difference rings $\mathbf{k}\subset F \subset K$, and such that every non zero divisor is a unit of $F$. Let $G$, be the real difference Galois group of (\ref{eq2}), $\mathcal{G}$ be the set of linear algebraic subgroups of $G$. 
\begin{enumerate}
\item For any $F\in \mathcal{F}$, the group $G(K/F)$ of elements of $G$ letting $F$ invariant belongs to $\mathcal{G}$.
\item For any $H\in \mathcal{G}$, the ring $K^{H}:=\{k\in K|\forall \f\in H, \f(k)=k\}$ belongs to $\mathcal{F}$.
\item Let $\a: \mathcal{F}\rightarrow \mathcal{G}$ and $\b: \mathcal{G}\rightarrow \mathcal{F}$ denote the maps $F\mapsto G(K/F)$ and $H\mapsto K^{H}$. Then, $\a$ and $\b$ are each other's inverses.
\end{enumerate}
\end{theo} 

\pagebreak[3]
\begin{rem}
If we replace $G$ by $G_{R}$, see the proof of proposition \ref{propo2}, which is a more natural candidate for the definition of the real difference Galois group, we loose the Galois correspondence. Take for example $\phi Y(x):=Y(x+1)=\exp(1)Y(x)$, which has solution $\exp(x)$. A real Picard-Vessiot extension for $Y(x+1)=\exp(1)Y(x)$ over $(\R,\phi)$ is $(\R[\exp(x),\exp(-x)],\phi)$. Let $K:=\R(\exp(x))$ be the total ring of fractions. We have $G\simeq \C^{*}$ and $G_{R}\simeq \R^{*}$. Note that $G_{R}\subset \mathrm{GL}_{1}(\R)$, viewed as a linear algebraic subgroup of $\mathrm{GL}_{1}(\C)$, equals to $G$. On the other hand, we have no bijection with the linear algebraic subgroups of $G_{R}$, which are $\{1\},\Z/2\Z$, $\R^{*}$, and the difference subfields of $K$, which are $\R(\exp(kx))$, $k\in \N$.
\end{rem}

\begin{proof}[Proof of Theorem \ref{theo2}]
 Let  $\mathcal{F}_{[i]}$ be the set of difference rings $\mathbf{k}[i]\subset F \subset K[i]$, such that every non zero divisor is a unit of $F$. Let $\mathcal{G}_{[i]}$ be the set of linear algebraic subgroups of $G_{K[i]}$. Remind that the field of constants of $\mathbf{k}[i]$ is algebraically closed. In virtue of the Galois correspondence in difference Galois theory, see \cite[Theorem 1.29]{VdPS97}, we find that
\begin{trivlist}
\item (a) For any $F\in \mathcal{F}_{[i]}$, the group $G_{K[i]}(K[i]/F)$ of elements of $G_{K[i]}$ letting $F$ invariant belongs to $\mathcal{G}_{[i]}$.
\item (b) For any $H\in \mathcal{G}_{[i]}$, the ring $K[i]^{H}$ belongs to $\mathcal{F}_{[i]}$.
\item (c) Let $\a_{[i]}: \mathcal{F}_{[i]}\rightarrow \mathcal{G}_{[i]}$ and $\b_{[i]}: \mathcal{G}_{[i]}\rightarrow \mathcal{F}_{[i]}$ denote the maps $F_{[i]}\mapsto G_{K[i]}(K[i]/F)$ and $H\mapsto K[i]^{H}$. Then, $\a_{[i]}$ and $\b_{[i]}$ are each other's inverses.
\end{trivlist}
We use Proposition \ref{propo3} to find that we have a bijection $\g
: \mathcal{F}\rightarrow \mathcal{F}_{[i]}$ given by $\g(F):=F[i]$. The inverse is $\g^{-1}(F)=F\cap K$. Now, let us remark that since the fundamental solution has coefficients in $R$, for all $F\in \mathcal{F}$,  $G(K/F)=G_{K[i]}(K[i]/\g(F))$. If we combine this fact with (a) and Proposition \ref{propo2}, we find (1).\par 
Proposition \ref{propo2} tells us that we may identify the groups in $ \mathcal{G}$ with the corresponding groups in $\mathcal{G}_{[i]}$.  To prove the point (2), we remark that for all $H\in \mathcal{G}$, $K[i]^{H}=K^{H}[i]$. Combined with (b), this shows the point (2) since $K^{H}=\g^{-1}(K[i]^{H})\in \mathcal{F}$ .\par 
 The point (3) follows from (c) and the fact that for all $F\in\mathcal{F}$ (resp. $H\in \mathcal{G}$) we have $G(K/F)=G_{K[i]}(K[i]/\g(F))$ (resp. $K[i]^{H}=\g(K^{H})$).
\end{proof}

\fussy

\pagebreak[3]

\bibliographystyle{alpha}
\bibliography{biblio}

\def\udot#1{\ifmmode\oalign{$#1$\crcr\hidewidth.\hidewidth
  }\else\oalign{#1\crcr\hidewidth.\hidewidth}\fi} \def\cprime{$'$}
  \def\polhk#1{\setbox0=\hbox{#1}{\ooalign{\hidewidth
  \lower1.5ex\hbox{`}\hidewidth\crcr\unhbox0}}} \def\cprime{$'$}
\begin{thebibliography}{CHvdP16}

\bibitem[BB62]{BiaBi}
A.~Bialynicki-Birula.
\newblock On {G}alois theory of fields with operators.
\newblock {\em Amer. J. Math.}, 84:89--109, 1962.

\bibitem[CH15]{CH13b}
Teresa Crespo and Zbigniew Hajto.
\newblock Real {L}iouville extensions.
\newblock {\em Comm. Algebra}, 43(5):2089--2093, 2015.

\bibitem[CHS08]{CHS}
Zo{\'e} Chatzidakis, Charlotte Hardouin, and Michael~F. Singer.
\newblock On the definitions of difference {G}alois groups.
\newblock In {\em Model theory with applications to algebra and analysis.
  {V}ol. 1}, volume 349 of {\em London Math. Soc. Lecture Note Ser.}, pages
  73--109. Cambridge Univ. Press, Cambridge, 2008.

\bibitem[CHS13]{CHS13}
Teresa Crespo, Zbigniew Hajto, and El{\.z}bieta Sowa.
\newblock Picard-{V}essiot theory for real fields.
\newblock {\em Israel J. Math.}, 198(1):75--89, 2013.

\bibitem[CHvdP16]{CHvdPa}
Teresa Crespo, Zbigniew Hajto, and Marius van~der Put.
\newblock Real and p-adic {P}icard--{V}essiot fields.
\newblock {\em Math. Ann.}, 365(1-2):93--103, 2016.

\bibitem[Coh65]{Cohn}
Richard~M. Cohn.
\newblock {\em Difference algebra}.
\newblock Interscience Publishers John Wiley \& Sons, New York-London-Sydeny,
  1965.

\bibitem[Dyc05]{Dyc}
Tobias Dyckerhoff.
\newblock Picard-vessiot extensions over number fields.
\newblock Fakultat fur Mathematik und Informatik der Universitat Heidelberg,
  diplomarbeit, 2005.

\bibitem[Fra63]{Fr63}
Charles~H. Franke.
\newblock Picard-{V}essiot theory of linear homogeneous difference equations.
\newblock {\em Trans. Amer. Math. Soc.}, 108:491--515, 1963.

\bibitem[HS08]{HS}
Charlotte Hardouin and Michael~F. Singer.
\newblock Differential {G}alois theory of linear difference equations.
\newblock {\em Math. Ann.}, 342(2):333--377, 2008.

\bibitem[Lam84]{Lam}
T.~Y. Lam.
\newblock An introduction to real algebra.
\newblock {\em Rocky Mountain J. Math.}, 14(4):767--814, 1984.
\newblock Ordered fields and real algebraic geometry (Boulder, Colo., 1983).

\bibitem[Mor09]{Morik}
Shuji Morikawa.
\newblock On a general difference {G}alois theory. {I}.
\newblock {\em Ann. Inst. Fourier (Grenoble)}, 59(7):2709--2732, 2009.

\bibitem[MU09]{MorUm}
Shuji Morikawa and Hiroshi Umemura.
\newblock On a general difference {G}alois theory. {II}.
\newblock {\em Ann. Inst. Fourier (Grenoble)}, 59(7):2733--2771, 2009.

\bibitem[Sei58]{Sei58}
A.~Seidenberg.
\newblock Abstract differential algebra and the analytic case.
\newblock {\em Proc. Amer. Math. Soc.}, 9:159--164, 1958.

\bibitem[vdPS97]{VdPS97}
Marius van~der Put and Michael~F. Singer.
\newblock {\em Galois theory of difference equations}, volume 1666 of {\em
  Lecture Notes in Mathematics}.
\newblock Springer-Verlag, Berlin, 1997.

\end{thebibliography}
\end{document}